 \documentclass[reqno,b5paper]{amsart}
\usepackage{amsthm}
\usepackage{enumerate}
\usepackage[mathscr]{eucal}
\newtheorem{theorem}{Theorem}[section]
\newtheorem{proposition}[theorem]{Proposition}
\newtheorem{lemma}[theorem]{Lemma}
\newtheorem{remark}[theorem]{Remark}
\newtheorem{definition}[theorem]{Definition}

\newcommand{\bz}{\mathbb{Z}}

\newcommand{\br}{\mathbb{R}}
\newcommand{\bc}{\mathbb{C}}

\newcommand{\bs}{\mathbb{S}}

\newcommand{\lr}{\longrightarrow}

\newcommand{\wt}{\widetilde}

\renewcommand{\span}{\textrm{span}}

\newcommand{\pic}{\textrm{Pic}}
\begin{document}
\baselineskip=15.5pt
\title[Quotients of complex Stiefel manifolds]{Vector fields on certain quotients of complex Stiefel manifolds} 
\author[S. Gondhali \and P. Sankaran]{Shilpa Gondhali* \and Parameswaran Sankaran**}

\newcommand{\acr}{\newline\indent}
\address{\llap{*\,}Tata Institute of Fundamental Research,
Mumbai, 
INDIA.}
\email{shilpa@math.tifr.res.in}
\address{\llap{**\,}Institute of Mathematical Sciences,
Chennai,
INDIA}
\email{sankaran@imsc.res.in}

\thispagestyle{empty}
\date{}

\begin{abstract}
 We consider quotients of complex 
Stiefel manifolds by finite cyclic groups whose action is induced by the scalar multiplication on the 
corresponding complex vector space.  
We obtain a description of their tangent bundles, compute their mod $p$ cohomology and 
obtain estimates for their span (with respect to their standard differentiable structure). 
We compute the Pontrjagin and Stiefel-Whitney classes of these manifolds and give applications to 
their stable parallelizability. 
\end{abstract}

\maketitle

\section{Introduction}
Let $W_{n,k}, 1\leq k<n,$ denote the complex Stiefel manifold of 
unitary $k$-frames $(v_1,\ldots,v_k)$ in $\bc^n$ where 
it is understood that $\bc^n$ has the standard hermitian 
metric.  One has the identification $W_{n,k}=U(n)/ U(n-k)$ where $U(n)$ denotes the group of unitary transformations 
of $\bc^n$ and $U(n-k)$ is imbedded in $U(n)$ 
as the subgroup that fixes the first $k$ standard basis vectors $e_1,\dots, e_k\in \bc^n$.  

One also has the complex projective Stiefel manifold $PW_{n,k}$
defined as the quotient of $W_{n,k}$ modulo the free action of the circle group $\bs^1$ which acts via scalar 
multiplication: $z(v_1,\ldots,v_k)=(zv_1,\ldots,zv_k)$ 
for $(v_1,\ldots,v_k)\in W_{n,k}$ and $z\in \bs^1$.   
Note that $PW_{n,k}=U(n)/(\bs^1\times U(n-k))$ where $\bs^1=\{z\in \bc\mid |z|=1\}$ is identified with the centre 
of $U(n)$. Observe that $\bs^1\times U(n-k)= U(1)\times
 U(n-k)\subset U(k)\times U(n-k)\subset U(n)$ where $U(1)\subset U(k)$ is the centre of $U(k)$ and $U(k)\times U(n-k)$ 
is the subgroup of $U(n)$ that stabilizes the 
complex vector subspace $\bc^k$, spanned by $e_1,\dots,e_k$.  Thus we get an equivalent description 
$PW_{n,k}=U(n)/(U(1)\times U(n-k))$. 

We define $W_{n,k;m}$ to be the quotient of $W_{n,k}$ by the subgroup $\Gamma_m\subset \bs^1$ of $m$-\textit{th} 
roots of unity.  Thus $\dim W_{n,k;m}=\dim W_{n,k}
=k(2n-k)$. The manifolds $W_{n,k;m}$ will be referred to as the $m$-\textit{projective Stiefel manifolds}. 
Clearly $W_{n,k;m}$ is the coset space $U(n)/(\Gamma_m\times U(n-k))$ and the obvious quotient map 
$W_{n,k;m}\lr PW_{n,k}$ is the projection of a principal bundle with
 fibre and structure group $\bs^1/\Gamma_m\cong \bs^1$. Also the projection $W_{n,k}\lr W_{n,k;m}$ is 
a covering map with deck transformation group $\Gamma_m$. 
In particular $\pi_1(W_{n,k;m})\cong \Gamma_m$ and the Euler characteristic $\chi(W_{n,k;m})$ vanishes.  
The manifold $W_{n,k;m}$ is orientable since 
$\Gamma_m$ is a subgroup of the connected group $\bs^1$ which reserves the orientation on $W_{n,k}$.    
Our aim in this paper is to initiate the study of the topology of $W_{n,k;m}$.    In \S2 we describe their tangent bundle 
and give (in Theorem \ref{span}) estimates for their span and stable span.  Span and other related notions 
will be recalled in \S2; see also \cite{kz94}.
We compute, in \S3,  the mod $p$ cohomology of $W_{n,k;m}$.   We also determine the height of the 
generator of $H^2(W_{n,k;m};\bz)\cong \bz_m$.  We show that, given $n,k$ where $1\leq k<n-1$,   
$W_{n,k;m}$ is not stably parallelizable 
for all but finitely many values of $m$. See Theorem \ref{sparallelizability} for the precise statement.
When $k=n-1$, $W_{n,n-1;m}$ is parallelizable, since $W_{n,n-1}\cong SU(n-1)$. 

The case $k=1$ corresponds to the (standard) lens space $L^n(m)=\bs^{2n-1}/\Gamma_m$.   
The non-parallelizability of spheres $W_{n,1}=\bs^{2n-1}, n\neq 1,2,4$, 
already implies non-parallelizability of the lens spaces $L^n(m)$ for any $m$.  Kambe's \cite[\S4]{kambe} 
result on immersion dimension for $L^n(p), p$ an odd prime, 
and the fact that $L^n(2)=\br P^{2n-1}$ are not stably 
parallelizable except when $n=1,2,4,$ implies that `most' of $L^n(m)$, $m>1$, are not \textit{stably} parallelizable.  
From the celebrated work of Adams, we know that 
$\span(L^n(m))\leq \span(\bs^{2n-1})=\rho(2n)-1$, where $\rho(n)$ is the Radon-Hurwitz number, defined as 
$\rho((2c+1)2^{4a+b})=8a+2^b$, where $a,c\geq 0$ and $0\leq b\leq 3$. See also \cite{iwata} for lower bounds 
for span of lens spaces.

In view of this, we assume that $1<k<n$ leaving out 
the case of lens spaces from consideration for the most part. 

Our proofs involve standard arguments making use of well-known results and techniques. 
The description of the tangent bundle of $W_{n,k;m}$ relies on the description of the tangent bundle of 
$PW_{n,k}$ due to Lam \cite{lam1}.  Estimates for (stable) span 
involve well-known arguments such as those employed in the context of real projective Stiefel manifolds; 
see \cite{kz94}, \cite{kz96}.  The cohomology calculations involve
 spectral sequences and known results 
concerning the cohomology of Stiefel manifolds and of projective Stiefel manifolds (see \cite{borel} 
and \cite{agmp}).

\section{The tangent bundle of $W_{n,k;m}$}
We describe below certain canonical vector bundles over 
the manifold $W_{n,k;m}$ and establish relations among them. We shall describe its tangent bundle and obtain 
lower bounds for their span and stable span.    \\

Let $1\leq k<n$ and let $m\geq 2$.
Let $\Gamma_m\subset U(1)$ denote the group of 
$m$-th roots of unity. 
Let $\pi_m:W_{n,k;m}\lr PW_{n,k}$ and $ \pi_1:W_{n,k} \lr PW_{n,k}$ be the canonical quotient maps.  
These are projections of principal bundles with structure groups 
$U(1)/\Gamma_m$ and $U(1)$ respectively.  
Let $p_m: W_{n,k}\lr W_{n,k;m}$ be the quotient map which is the universal covering projection with deck  
transformation group $\Gamma_m.$  One also has the obvious 
covering projections $p_{m,l}\colon W_{n,k;l}\lr W_{n,k;m}$ whenever $l|m$.   
Note that $\pi_1=\pi_m\circ p_m$ and $p_m=p_{m,l}\circ p_l$.   We shall denote by
 $[v_1,\ldots,v_k]_m$ (or simply $[v_1,\ldots,v_k]$ when there is no danger of confusion) the element 
$\pi_m(v_1,\ldots,v_k)\in W_{n,k;m}$ where $(v_1,\ldots,v_k)\in W_{n,k}$.   Also, 
$\pi_1(v_1,\ldots,v_k)\in PW_{n,k}$ will be denoted $[v_1,\ldots, v_k]_0$ (or more briefly $[v_1,\ldots, v_k]$).  \\

Let $\zeta_{n,k}$ denote the complex line bundle over 
$PW_{n,k}$ associated to the principal $U(1)$-bundle $\pi_1:W_{n,k}\lr PW_{n,k}$.  Thus, 
the total space of $\zeta_{n,k}$ is the fibre product 
$W_{n,k}\times_{U(1)}\bc$.  It is isomorphic to the 
bundle over $PW_{n,k}$ whose fibre over a point $[v_1,\ldots, v_k]_0$ is the complex vector space $\bc v_1\subset \bc^n$.  
Define $\xi_{n,k;m}:=\pi_m^*(\zeta_{n,k})$ and let $\gamma_{n,k;m}$ be 
the complex line bundle associated to the principal $U(1)$-bundle obtained by extension of structure group 
via the character $\Gamma_m\subset U(1)$ of the 
$\Gamma_m$-bundle $W_{n,k}\lr W_{n,k;m}$. 
(When $m=2$, $\xi_{n,k;m}$ is the complexification of the real line bundle associated to the double cover 
$W_{n,k}\lr W_{n,k;2}$.)
Explicitly, $\gamma_{n,k;m}$ has 
total space the twisted product $W_{n,k}\times_{\Gamma_m}\bc$ where $\Gamma_m$ operates on $\bc$ by scalar multiplication.  
We have the following lemma. 
We outline a proof, which is elementary, as the lemma will be used throughout.  For any vector bundle 
$\eta$, $\eta^l$ denotes the $l$-fold tensor product 
with itself and $l\eta$, the $l$-fold Whitney sum with 
itself.  

\begin{lemma}\label{pullback}
(i)  The complex line bundle associated to the 
principal $U(1)/\Gamma_m\cong U(1)$-bundle with projection $W_{n,k;m}\lr PW_{n,k}$ is isomorphic 
to $\zeta_{n,k}^m$.\\
(ii) One has an isomorphism $\xi_{n,k;m}\cong \gamma_{n,k;m}$ of complex line bundles over $W_{n,k;m}$. 
\end{lemma}

\begin{proof}
(i) The isomorphism $U(1)/\Gamma_m\lr U(1)$ is induced 
by the homomorphism $z\mapsto z^m$ of $U(1)$ onto 
itself.  This homomorphism induces the map 
$\eta\mapsto \eta^m$, for any line bundle $\eta$ associated to a principal 
$U(1)$-bundle.  By definition, $\zeta_{n,k}$ is associated to the principal $U(1)$-bundle  
$\pi_1:W_{n,k}\lr PW_{n,k}$.  Since $\pi_m:W_{n,k;m}\lr PW_{n,k}$ is the $U(1)/\Gamma_m$-bundle associated to 
$\pi_1$, it follows that the complex line bundle associated
 to the principal $U(1)/\Gamma_m\cong U(1)$-bundle $\pi_m$ is $\zeta_{n,k}^m$.\\

(ii) By the very definition of $\gamma_{n,k;m}$, its total space has the description 
$E(\gamma_{n,k;m})=\{[x,z]\mid x\in W_{n,k}, z\in \bc\}$ where $[x,z]=[x',z']$ if 
and only if $xg=x', g^{-1}z=z'$ for some $g\in \Gamma_m$.
Also, one has $E(\xi_{n,k})=\{(p_m(v),tv_1)\mid 
t\in \bc, x=(v_1,\ldots,v_k)\in W_{n,k}\}$. Consider 
the map $f:E(\xi_{n,k;m})\lr E(\gamma_{n,k;m})$ defined as $(p_m(v),tv_1)\mapsto [v;t]$.  It is readily 
checked that this is a well-defined continuous map that covers the identity map 
of the base space $W_{n,k;m}$, and is a linear isomorphism on each fibre. This completes the proof.    \end{proof}

Observe that $\xi_{n,k;m}^l$, which corresponds to the character $\Gamma_m\lr U(1), z\mapsto z^l,$ is non-trivial 
when $1\leq l<m$.   In particular it follows 
that the order of the class of $\xi_{n,k;m}$ in the Picard group $\pic(W_{n,k;m})$ of $W_{n,k;m}$ is $m$.  
Indeed, $\xi_{n,k;m}$ is a generator of $\pic(W_{n,k;m})
\cong \bz_m.$  For, one has $H_1(W_{n,k;m};\bz)\cong \pi_1(W_{n,k;m})\cong \bz_m$.
It is not difficult to see that $H_2(W_{n,k;m};\bz)=0$. Hence $\pic(W_{n,k;m})\cong H^2(W_{n,k;m};\bz)\cong \bz_m$ 
by the universal coefficient theorem.  
 The projection $\pi_m:W_{n,k;m}\lr PW_{n,k}$ induces a surjection $H^2(PW_{n,k};\bz)\cong \bz\lr \bz_m\cong 
H^2(W_{n,k;m};\bz)$ and hence maps the generator $c_1(\zeta_{n,k})$ to the 
generator of $\bz_m$. By the naturality of  Chern classes  
we see that $c_1(\xi_{n,k;m})$ is a generator of $H^2(W_{n,k;m};\bz)\cong \bz_m$.  Summarising we 
have

\begin{lemma}\label{picard} The Picard group 
{\em $\pic(W_{n,k;m})$} is isomorphic to $\bz_m$ and is generated by $\xi_{n,k;m}$. 
\end{lemma}

The following isomorphism of complex vector bundles on 
$PW_{n,k}$ 
is well-known and is due to K.-Y. Lam \cite{lam1}:
$k\zeta_{n,k}\oplus \beta_{n,k}\cong n\varepsilon_\bc$ where $\varepsilon_\bc$ denotes the trivial complex line 
bundle and $\beta_{n,k}$ is the complex $(n-k)$-plane 
bundle whose fibre over $[v_1,\ldots,v_k]_0$ is the vector space $\{v_1,\ldots,v_k\}^\perp\subset \bc^n$ where the 
orthogonal complement is taken with respect to the 
standard hermitian inner product on $\bc^n$.   
Pulling back to $W_{n,k;m}$ via the projection $\pi_{n,k;m}$ we obtain an isomorphism 

\[k\xi_{n,k;m}\oplus \beta_{n,k;m}\cong n\varepsilon_\bc\eqno(1)\]
of complex vector bundles over  $W_{n,k;m}$ where $\beta_{n,k;m}:=\pi_m^*(\beta_{n,k})$. 
Tensoring with the dual bundle $\xi_{n,k;m}^\vee\cong 
\xi_{n,k;m}^{m-1}$ we get
$k\varepsilon_\bc\oplus \xi_{n,k;m}^\vee\otimes_\bc\beta_{n,k;m}\cong n\xi_{n,k;m}^\vee=n\xi_{n,k;m}^{m-1}.$
Taking duals, we obtain 
\[k\varepsilon_\bc\oplus \xi_{n,k;m}\otimes_\bc\beta_{n,k;m}^\vee\cong n\xi_{n,k;m}. \eqno(2)\]

Recall from \cite[Theorem 3.2]{lam1} that the tangent bundle $\tau PW_{n,k}$ of $PW_{n,k}$ is isomorphic to the 
(real) vector bundle $k\zeta_{n,k}^\vee\otimes_\bc
 \beta_{n,k}\oplus (k^2-1)\varepsilon_\br$. Since $\pi_{m}\colon W_{n,k;m}\lr PW_{n,k}$ is a \textit{principal} 
$\bs^1$-bundle, we have 
\[\tau W_{n,k;m}\cong k\xi_{ n,k;m}^\vee \otimes_\bc \beta_{n,k;m}\oplus k^2\varepsilon_\br. \eqno(3)\] 
In the above isomorphism, and in the sequel, we have used the same symbol to denote  
a complex vector bundle and its underlying real vector 
bundle, as there is no risk of confusion.

\begin{remark} {\em Assume that $k$ is even, equivalently 
$W_{n,k;m}$ is even dimensional.
Then $\tau W_{n,k;m}$ 
has a complex structure arising from the isomorphism of vector bundles given in (3). Thus $W_{n,k;m}$ admits an 
almost complex structure.  Recall that,  by the work of Wang \cite{wang},  
$W_{n,k}=SU(n)/SU(n-k)$ admits a complex structure invariant under the left action of $SU(n)$. When $m$ divides 
$n$, $\Gamma_m$ 
is contained in the centre $\Gamma_n$ of $SU(n)$. In this case the action of $\Gamma_m$ on $W_{n,k}$ preserves 
the complex structure.  We conclude that 
$W_{n,k;m}$ admits a complex structure when it is even dimensional and $m|n$. 
}
\end{remark}

Using the isomorphism (1) and the fact that $\varepsilon_\bc=2\varepsilon_\br$, we obtain an isomorphism 
\[\tau W_{n,k;m}\oplus k^2\varepsilon_\br
\cong k(\xi_{n,k;m}^\vee\otimes _\bc\beta_{n,k;m}\oplus k\varepsilon_\bc)=nk\xi_{n,k;m}^\vee \eqno(4)\]
of real vector bundles.

Recall that the \textit{span} of a smooth manifold $M$ is the maximum number $r\geq 0$ for which there exist 
$r$ everywhere linearly independent vector fields on $M$. Equivalently 
span of $M$ is the maximum number $r$ such that 
$\tau M\cong r\varepsilon_\br\oplus \eta$ for some vector bundle $\eta$. The \textit{stable span} of $M$ is the maximum 
number $s$ such that $\tau M\oplus t\varepsilon_\br\cong (s+t)\varepsilon_\br\oplus \theta$ for some vector bundle 
$\theta$ where $t>0$.  Indeed one may take $t=1$ in
 the above definition of stable span.  The rank of $\theta$ is then called the \textit{geometric dimension} of 
$\tau M$. We denote the span of $M$ by $\span(M)$.  The notions of span, stable span, and geometric dimension 
can be extended in an obvious manner to any vector bundle. 
 The reader may  refer to \cite{kz94} and \cite{kz96} for a detailed discussion on the vector field 
problem, which asks for the determination of the span of a given smooth manifold.

\begin{theorem} \label{span}
Suppose that $2\leq k<n$ and $m\geq 2$.  Then:\\
(i)  $\textrm{{\em span}}(W_{n,k;m})>\textrm{{\em stable span}}(PW_{n,k})\geq \dim(W_{n,k;m})-2n+1$.
Moreover, when $n$ is even, $\textrm{{\em span}}(W_{n,k;m})>\dim (W_{n,k;m})-2n+3.$\\
(ii)  $\textrm{{\em span}}(W_{n,k;m})>\textrm{{\em stable span}}(W_{n,k-1;m})$.\\ 
(iii) $W_{n,n-1;m}$ is parallelizable.
\end{theorem} 

\begin{proof} (i) 
Since $\pi_m$ is a principal $\bs^1$-bundle, one has the bundle isomorphism 
$\tau (W_{n,k;m})\cong \pi^*(\tau PW_{n,k})\oplus \varepsilon_\br=\pi^*(\tau PW_{n,k}\oplus \varepsilon_\br)$. 
Hence $\span(W_{n,k;m})>\textrm{stable span}(PW_{n,k})$.
 Now consider the projection $q:PW_{n,k}\lr 
\bc P^{n-1}$ defined as $[v_1,\ldots,v_k]\mapsto [v_1]$.  The stable tangent bundle $\tau PW_{n,k}\oplus 
(k^2+1)\varepsilon_\br$ is isomorphic 
to $nk\zeta_{n,k}=q^*(nk\zeta_{n,1})$.  Clearly the bundle $nk\zeta_{n,1}$ over $\bc P^{n-1}$ contains 
a trivial real vector bundle of rank  
$2(nk-(n-1))$.  (See \cite{husemoller}.) Therefore the stable span of $PW_{n,k}$ 
is at least $2nk-2(n-1)-(k^2+1)=\dim W_{n,k;m}-2n+1.$

Let $n$ be even. The complex $2$-plane bundle $\zeta_{n,1}\oplus \zeta_{n,1}^\vee$ evidently admits a 
reduction of structure group to $SU(2)=Sp(1)$. 
Hence it is the underlying complex vector bundle of a 
quaternionic line bundle.  Any such bundle can be classified by a map into the 
quaternionic projective space $\mathbb{H}P^r$ where $r 
=\lfloor (1/4)\dim_\br(\bc P^{n-1})\rfloor=n/2-1$.  That is, 
there exists a continuous map $h:\bc P^{n-1}\lr \mathbb{H}P^{r}$ such that $h^*(\omega)\cong 
\zeta_{n,1}\oplus \zeta_{n,1}^\vee$ where $\omega$ is the canonical quaternionic line bundle over $\mathbb{H} P^r$. 
 The underlying real vector bundle $(nk/2)\omega$
 admits $(2nk-4r)\varepsilon_\br$ as a summand and so, working with underlying real vector bundles throughout, we have  
$nk\zeta_{n,1}= (nk/2)(q^*(\zeta_{n,1}\oplus\zeta_{n,1}^\vee))= h^*((nk/2 ).\omega)=(2nk-4r)\varepsilon_\br\oplus\eta$ 
for some real vector bundle $\eta$. 
 As before, it follows that $\textrm{span}(W_{n,k;m})>\textrm{stable span}(PW_{n,k})\geq \dim W_{n,k;m}-2n+3$.

(ii) Consider the fibre bundle projection $W_{n,k}\lr W_{n,k-1}$ with fibre $\bs^{2n-2k+1}$. 
Since it is $\Gamma_m$-equivariant, we obtain a $\bs^{2n-2k+1}$-bundle with projection $p:W_{n,k;m}\lr W_{n,k-1;m}$.  
 Note that $p^*(\xi_{n,k-1;m})=\xi_{n,k;m}$ and $p^*(\beta_{n,k-1;m})=\beta_{n,k;m}\oplus \xi_{n,k;m}$. 
Write $k^2\varepsilon_\br$ as $(k-1)^2\varepsilon_\br\oplus(k-1)(\xi^\vee_{n,k;m}\otimes_\bc \xi_{n,k;m})\oplus\varepsilon_\br$. Substituting this in the expression 
(3) for $\tau W_{n,k;m}$ and observing that $k\xi_{n,k;m}^\vee\otimes\beta_{n,k;m}\oplus (k-1)\xi_{n,k;m}^\vee\otimes_\bc
\xi_{n,k;m}=(k-1)\xi_{n,k;m}^\vee\otimes 
(\beta_{n,k;m}\oplus \xi_{n,k;m})\oplus \xi_{n,k;m}^\vee \otimes \beta_{n,k;m}$, we obtain that 
$\tau W_{n,k;m}\cong p^*(\tau W_{n,k;m-1})\oplus \varepsilon_\br\oplus \xi^\vee_{n,k;m}\otimes \beta_{n,k;m}\cong 
p^*(\tau W_{n,k;m-1}\oplus \varepsilon_\br)\oplus
 \xi^\vee_{n,k;m}\otimes \beta_{n,k;m}$. 
Therefore $\span (W_{n,k;m})\geq \span(\tau W_{n,k-1;m}\oplus \varepsilon_\br)>\textrm{stable span}(W_{n,k-1;m})$ as asserted. 

(iii) Note that $W_{n,n-1}\cong SU(n)$.  Therefore $W_{n,n-1;m},$ being a quotient of a Lie group by a 
finite subgroup, is parallelizable. 
\end{proof}

We refer the reader to \cite{yoshida} and \cite{iwata} for the span of lens spaces.

\begin{proposition}\label{stablespan} Let $2\leq  k<n$ and let $m\geq 2$.
One has 
\[\textrm{{\em span}}(W_{n,k;m})=\textrm{{\em stable span}}(W_{n,k;m})\] in each of the following cases: 
(i) $k$ is even, (ii) $n$ is odd, and, (iii) $n\equiv 2 \mod 4$.  
\end{proposition}
\begin{proof}  From (3) we obtain that $\span(W_{n,k;m})\geq k^2\geq 4$.  Since $W_{n,k;m}$ is 
orientable, the first Stiefel-Whitney class  $w_{1}(W_{n,k;m})$ vanishes.  As observed already in the introduction, 
the Euler characteristic $\chi(W_{n,k;m})$ vanishes.
 Furthermore, it follows from (4) that the Stiefel-Whitney classes $w_1(W_{n,k;m}),w_2(W_{n,k;m})$ vanish when 
$nk$ is even.  (We shall give a formula for the total
 Stiefel-Whitney class of $W_{n,k;m}$ in 
Proposition \ref{pontrjagin}, which also implies that $w_i(W_{n,k;m})=0, i=1,2$ when $nk$ is even.)

Our hypotheses on $k$ and $n$ imply that 
$d=\dim(W_{n,k;m})=k(2n-k)$ is even when $k$ is even, $d\equiv 1\mod 4$ when both $k$ and $n$ are  odd, 
and $d\equiv 3 \mod 8$ when $k$ is odd and $n\equiv 2 \mod 4$. 
 The proposition now follows from the work of Koschorke.  More precisely,  (i) follows from 
\cite[Theorem 20.1]{koschorke} and, 
assuming, as we may, that $k$ is odd, (ii) and (iii) follow, respectively, from  Corollaries 20.9, 
and 20.10 of \cite{koschorke}. \end{proof}

\begin{remark}{\em Recall that the generalized vector 
field problem asks  for the determination of the geometric dimension of multiples of the Hopf bundle  
$\xi_{n}$ over the real projective space $\mathbb{R}P^{n}$. When $m=2$,  $W_{n,1;2}=\mathbb{R}P^{2n-1}$ 
and the bundle $\xi_{n,1;2}$ is isomorphic, as a real vector bundle,
 to $2\xi_{n-1}$.  Denoting by $p: W_{n,k;2}\longrightarrow \mathbb{R}P^{2n-1}$ the  projection 
$[v_1,\ldots,v_k]\mapsto [v_1]$ we see that 
$\xi_{n,k;2}\cong p^*(2\xi_{2n-1})$.   Therefore, 
using the bundle isomorphism (4), we have 
\[\textrm{stable span} (W_{n,k;2})\geq \textrm{span}(2nk\xi_{2n-1})-k^2. \eqno(5)\]  Invoking Proposition  
\ref{stablespan} we 
obtain the following lower bound:
\[\textrm{span}(W_{n,k;2})\geq \textrm{span}(2nk\xi_{2n-1})-k^2\eqno(5) \]
when $k$ is even, or $n$ is odd, or $n\equiv 2 \mod 4$.  
Although the generalized vector field problem is yet to be solved completely, the precise value of the span 
of $r\xi_{n}$ is known from the work of Lam \cite[Theorem 1.1]{lam2}
 when $r=8l+p,n=8m+q$, $l\geq m\geq 0$, ${l\choose m}$ is odd, $0\leq p,q\leq 7$.   See also \cite{dl}. 
In many cases,  (6)  yields a better lower bound than Theorem \ref{span}(i). }
\end{remark}

We conclude this section with the following 

\begin{proposition}\label{universal}
Let $m\geq 2$ be an integer. 
Let $X$ be any topological space and let $\xi$ be a complex line bundle over $X$ such that 
(i) $\xi$ admits a reduction of structure group to $\bz_m$, and, (ii) $n\xi $ admits $k$ everywhere 
$\mathbb{C}$-linearly independent cross-sections. Then there exists a continuous map $f:X\lr W_{n,k;m}$ 
such that $f^*(\xi_{n,k;m})\cong \xi$. 
\end{proposition}
\begin{proof}
Let $p:\wt{X}\lr X$ be a regular covering projection 
with deck transformation group $\Gamma_m\cong\bz_m$ such that $\xi$ is isomorphic to the bundle with projection 
$E:=\wt{X}\times_{\Gamma_m}\bc\lr X$ where $\Gamma_m$ acts on $\bc$ via a character $\Gamma_m\lr \bc^*$.  
The existence of such a covering is the content of (i).  
We identify the total space of $\xi$ with $E$.  
Observe that 
$\xi$ admits a hermitian metric: $(e,e')\mapsto z\bar{z}'$ 
is a hermitian metric where $e=[x,z],e'=[x,z']\in E$, $x\in X,~z,z'\in \bc$.  
Consequently $\xi^\vee$ also admits a hermitian metric.

In view of (ii) and the existence of a hermitian metric 
on $n\xi$, we have a splitting $n\xi\cong k\varepsilon_\bc\oplus \theta$.  
Taking duals, we get 
$n\xi^\vee\cong k\varepsilon_\bc\oplus \theta^\vee$. 
Tensoring with $\xi$, we see that 
$n\varepsilon_\bc=k\xi\oplus\eta$ where $\eta:=\xi\otimes\theta^\vee$. Then $\eta$ also admits a hermitian metric  
which is such that each copy of $\xi$ and $\eta$ are pairwise orthogonal.  

For any hermitian vector bundle $\nu$ of rank $n$ over $X$, one has an associated $W_{n,k;m}$-bundle, 
denoted $W_{n,k;m}(\nu)$, defined as the space of all $\Gamma_m$-equivalence 
classes of unitary $k$-frames in each fibre of $\nu$. 
When $\nu$ is trivial, this is just the product bundle $X\times W_{n,k;m}\lr X.$    

Now one has a cross-section $\sigma:X\lr X\times W_{n,k;m}=W_{n,k;m}(k\xi\oplus \eta)$ defined as follows:
For any $x\in X$, let $\wt{x}\in p^{-1}(x)$ be any point in the fibre over $x\in X$. We identify $\wt{x}$ with 
$[\wt{x},1]\in \wt{X}\times_{\Gamma_m}\bc=E(\xi)$. 
Then $\sigma(x)=[\wt{x},\ldots,\wt{x};0]\in W_{n,k;m}(k\xi\oplus \eta)$ is well-defined and is independent of 
the choice of $\wt{x}$ in $p^{-1}(x)$.  Since $p:\wt{X}\lr X$, is a local homeomorphism, it is immediate that 
$\sigma$ is continuous.  

The desired map $f:X\lr W_{n,k}$ is now obtained 
as the composition $pr_2\circ \sigma$. 
\end{proof}

\begin{remark}
{\em
An analogue of the above property for real projective 
Stiefel manifolds was established in \cite{bhlsz}.  A similar universal property 
for complex projective Stiefel manifolds was established 
in \cite{agmp}, under the additional assumption that 
$X$ be a finite CW complex.   
}
\end{remark}

\section{The mod $p$ cohomology}
In this section we shall describe the mod $p$ cohomology of $W_{n,k;m}$ where $p$ is a prime.   
Recall  that $H^*(W_{n,k};\bz)$ is isomorphic to the exterior algebra 
$\Lambda_{\bz}(v_{2n-2k+1}, \ldots, v_{2n-1})$ 
where $v_q\in H^q(W_{n,k};\bz).$  This result is attributed to C. Ehresmann by Borel \cite[Prop. 9.1]{borel}.
It is 
customary to denote by $\Lambda_{\bz_p}(x_1,\ldots,x_k)$  
any graded commutative algebra $A$ over $\bz_p$ 
in which 
square-free monomials in $x_1,\ldots, x_r$ form a basis.  
(If $p$ is odd, and all the generators $x_j$ have odd degree, then $A$ is isomorphic to the exterior algebra. 
However when $p=2$, it need not be so.)  This convention
 will be used in what follows. 

\noindent
\textit{Notations}:
Let $N:=2N'$ where 
$N':=N'_p=\min_{n-k+1\leq j\leq n}\{ j\mid {n\choose j}\not\equiv 0\mod p\}$.  
(Note that the value of $N'$ depends on $n,k$ and $p$.)  In what follows, 
we shall label (homogeneous) generators of a graded algebra by their degrees. Thus $|y_j|=j$ when 
$y_j \in H^*(X;R)$.

\begin{theorem}\label{cohomology}
Suppose that $2\leq k<n$ and $m\geq 2$.\\  
(i) If $p$ is any prime not dividing $m$,  then 
\[p^*_m:H^*(W_{n,k;m};\bz_p)\cong 
H^*(W_{n,k};\bz_p)=\Lambda_{\bz_p}(v_{2n-2k+1},\ldots,v_{2n-1})\] is an isomorphism of algebras.\\
(ii) If $p$ is a an odd prime that divides $m$, then 
$$\begin{array}{ll}
H^*(W_{n,k;m};\bz_p)&\cong H^*(\bs^1;\bz_p)\otimes 
H^*(PW_{n,k};\bz_p)\\~&\cong \bz_p[y_2]/\langle y_2^{N'}\rangle \otimes \Lambda_{\bz_p}
(y_1,y_{2n-2k+1}, y_{2n-2k+3},\ldots, \hat{y}_{N-1}, \ldots, y_{2n-1})
\end{array}$$
where $N, N'$ are as defined above.  (As usual, $\hat{}$ stands for omission of the variable.)  
Also $y_2=c_1(\xi_{n,k;m})\mod p$.\\
(iii)(a) Suppose $m\equiv 2 \mod 4$.  Then 
\[H^*(W_{n,k;m};\bz_2)
=\bz_2[y_1]/\langle y_1^{N}\rangle\otimes\Lambda_{\bz_2}(y_{2n-2k+1}, y_{2n-2k+3},\ldots, 
\hat{y}_{N-1}, \ldots, y_{2n-1}).\] 

(b) Suppose that $m\equiv 0 \mod 4$. Then 
\[H^*(W_{n,k;m};\bz_2)
=\bz_2[y_2]/\langle  y_2^{N'}\rangle\otimes\Lambda_{\bz_2}(y_1,y_{2n-2k+1}, y_{2n-2k+3},\ldots, 
\hat{y}_{N-1}, \ldots, y_{2n-1}),\] 
where $y_1^2=0$. 
\end{theorem}
\begin{proof}
(i) Let $\gamma\in \Gamma_m$.  Recall that 
$p_m:W_{n,k}\lr W_{n,k;m}$ is the covering projection with deck transformation group $\Gamma_m$. 
The covering map $\gamma\colon W_{n,k}\lr W_{n,k}$  
is homotopic to the identity since $\gamma\in U(n)$ and $U(n)$ is connected.  
It follows that $\Gamma_m$ 
acts trivially on the cohomology groups of $W_{n,k}$. 
Since $p$ does not divide $m$, $p_m^*:H^*(W_{n,k;m};\bz_p)\lr H^*(W_{n,k};\bz_p)^{\Gamma_m}=
H^*(W_{n,k};\bz_p)$ is an isomorphism. One knows that $H^*(W_{n,k};\bz_p)
\cong \Lambda_{\bz_p}(y_{2n-k+1},\ldots, y_{2n-1})$ (see \cite[Proposition 9.1]{borel}).  

(ii)  By definition, $\xi_{n,k;m}=\pi_m^*(\zeta_{n,k})$ where $\pi_m:W_{n,k;m}\lr PW_{n,k}$ is 
the projection of the principal $U(1)/\Gamma_m\cong \bs^1$-bundle.  
(See \S 2.) Let $y_2=c_1(\zeta_{n,k})\in H^2(PW_{n,k};\bz_p)$.
We apply the Serre spectral sequence with $\bz_p$-coefficients to the principal $\bs^1$-bundle 
with projection $\pi_m$. 
The differential $d:E^{0,1}_2\lr E^{2,0}_2$ maps the generator of $E^{0,1}_2\cong H^1(\bs^1;\bz_p)
\cong \bz_p$ to $c_1(\zeta_{n,k}^m)=my_2\in H^2(PW_{n,k};\bz_p)\cong
 \bz_p y_2$ by Lemma \ref{pullback} (i).  Since $p|m$, this differential is zero. It follows that 
the spectral sequence collapses and we get 
$H^*(W_{n,k;m};\bz_p)\cong H^*(\bs^1;\bz_p)\otimes 
H^*(PW_{n,k};\bz_p)$. Note that $y_1^2=0$ as 
$p$ is odd. 
The rest of the statement follows 
from the description of the $\bz_p$-cohomology of 
$PW_{n,k}$ due to Astey, Gitler, Micha and Pastor \cite{agmp}.
  
(iii) We proceed as in (ii) and obtain 
that $H^*(W_{n,k;m};\bz_2)\cong H^*(\bs^1;\bz_2)\otimes H^*(PW_{n,k};\bz_2)$ as an  
$H^*(PW_{n,k};\bz_2)$-module.
Denoting by $\nu$  
the real line bundle 
associated to the double cover $f:W_{n,k;l}\lr W_{n,k;m}$ where $m=2l$ we have $w_1(\nu)=:y_1$ is 
the generator of $H^1(W_{n,k;m};\bz_2)$.  Also, 
$\nu\otimes_\br\bc$ is evidently an element of order $2$ in $\pic(W_{n,k;m})$ and hence 
$\nu\otimes_\br \bc 
\cong \xi_{n,k;m}^l$ by Lemma \ref{picard}.  It follows that $y_1^2=c_1(\nu\otimes\bc)=
lc_1(\xi_{n,k;m})=ly_2$. Hence  
$y_1^2=0$ if $l$ is even, and $y_1^2=y_2$ if $l$ is odd. 

Finally, write $d=\dim W_{n,k;m}$. Then $H^{d-1}(PW_{n,k};\bz_2)\otimes \bz_2y_1\cong 
H^{d}(W_{n,k;m};\\ \bz_2)\cong \bz_2$.   
Therefore $y_1y_2^{N'-1}y_{2n-2k+1}\ldots \hat{y}_{N-1} \ldots y_{2n-1}$ generates
 $H^d(W_{n,k;m};\bz_2)\\ \cong\bz_2$.    
Using this, and the property that square-free monomials in $y_{2n-2k+1},\\ \ldots,  
\hat{y}_{N-1}, \ldots, y_{2n-1}$  
are linearly independent, it follows that the same property holds for   
$y_1, y_{2n-2k+1}, \ldots, \hat{y}_{N-1}, \ldots, y_{2n-1}$.  
This completes the proof. 
\end{proof}

We now turn to the integral cohomology of $W_{n,k;m}$.
It is easily seen that $H^2(W_{n,k;m};\bz)\cong \bz_m$ 
generated by $y_2=c_1(\xi_{n,k;m})$. We are mainly interested in the height of $y_2$. 
Recall that the \textit{height} of $0\neq y\in H^q(X;R)$ is the largest positive integer 
$h$ such that $y^{h}\neq 0$.  
In view of the fact 
that the complex Stiefel manifold $W_{n,k}$ is 
$2(n-k)$-connected, we see that the $2(n-k)$-skeleton 
of $W_{n,k;m}$ with respect to any CW-structure may be 
regarded as the $2(n-k)$-skeleton of the infinite 
lens space $L^\infty(m)$ with fundamental group $\bz_m$.   
So $H^q(W_{n,k;m};\bz)\cong H^q(L^\infty(m);\bz)\cong H^q(\bz_m;\bz)$ for $q<2(n-k)$.  It is well-known 
that $H^*(L^\infty(\bz_m);\bz)\cong \bz[y_2]/\langle 
my_2\rangle$; see \cite{hatcher}.  It follows that $H^q(W_{n,k;m};\bz)\cong \bz_my^r$ where $q=2r<2n-2k$.  
However, the following theorem gives the precise value of the height. 

\begin{definition}
Fix integers $n,k,m$ such that  $m>1$ and $1\leq k<n$. We define $m_r:=m$ if $r\leq n-k$ and 
$m_r:=\gcd\{m, {n\choose j};n-k<j\leq r\}$ if $n-k<r\leq n$.  
\end{definition}

The integral cohomology ring of lens spaces is well-known. We 
shall now establish the following

\begin{theorem}\label{height}  With the above notations, the (additive) order of $y_2^r\in H^{2r}
(W_{n,k;m};\bz)$ is $m_r$ for  $1\leq r\leq n$. In particular 
the height of $y_2\in H^2(W_{n,k;m}; \bz)\cong \bz_m$ is the largest integer $h$,  $n-k<h\leq n$, 
such that $m_h>1$. 
\end{theorem} 
\begin{proof} By our observation above, we need only consider the case $r>n-k$. 

Let $E$ be a contractible CW complex on which $\Gamma_m$ acts freely so that the quotient  
$E/\Gamma_m=K(\Gamma_m,1)$ has the same homotopy 
type as the infinite lens space $L^\infty (m)$.  
Then the fibre product $W':=E\times_{\Gamma_m} W_{n,k}$ 
fibres over $W_{n,k;m}$ with fibre $E$.  In particular, 
$W'$ has the same homotopy type as $W_{n,k;m}$.  Also 
one has a fibre bundle with fibre $W_{n,k}$ with projection $W'\lr K(\Gamma_m,1)$.  
We choose $E$ conveniently so that it is easier to 
determine the differential in the Serre spectral sequence 
associated to the $W_{n,k}$-bundle over $K(\Gamma_m,1)$. 

Let $E:=W_{\infty,n}=\bigcup_{r>n} W_{r,n}$ be the space
of all unitary $n$-frames in $\bc^\infty=\bigcup_{r>1}\bc^r$.   An element of $\bc^\infty$ is viewed as a column 
vector whose entries are eventually $0$. The space $E$ is contractible since the inclusion 
$W_{r,n}\subset W_{r+1,n}$ is $2(r-n)$-connected for all $r>n$. 

The group $\Gamma_m$ acts on $\bc^\infty$ via scalar multiplication and hence one has the diagonal 
action of $\Gamma_m$ on $W_{\infty,n}$. The quotient 
$W_{\infty,n}/\Gamma_m=:L$ has the homotopy type  
of the infinite lens space $L^\infty(m)$. The image of ${\bf v}=(v_1,\ldots,v_n)\in W_{\infty,n}$ in 
$L$ under the quotient map $q:W'\lr L$ will be denoted 
$[v_1,\ldots,v_n]_m$ or $ [{\bf v}]$.  The $n$-dimensional complex vector space $\bc v_1+\ldots+
\bc v_n$ will be denoted $\langle {\bf v}\rangle$. Denote by $W_k(V)$ the space of all unitary $k$-frames 
in the complex 
vector space $V\subset\bc^\infty$.

Let $W':=\{([{\bf v}]_m;u_1,\ldots, u_k)\mid {\bf v}\in W_{\infty,n}, (u_1,\ldots,u_k)\in W_k(\langle v\rangle)\}.$ 
The space $W'$ is just the $W_{n,k}$-bundle over $L$ associated to the $n$-plane bundle 
$n\gamma$ where $\gamma$ is the complex line bundle 
associated to the character $\pi_1(L)=\Gamma_m\subset \bs^1$.  Let $\wt{f}:W'\lr E/U(n-k)=BU(n-k)$ 
be defined as $\wt{f}([{\bf v}]_m;u_1,\ldots, u_k)=
(\langle{\bf v}\rangle;u_1,\ldots,u_k)$.  We let $f:L\lr E/U(n)=BU(n)$ be the map 
$[{\bf v}]_m\mapsto \langle{\bf v}\rangle$.   
One has the projection $\pi: BU(n-k)\lr BU(n)$, defined 
as $(\langle {\bf u}\rangle; u_1,\ldots, u_k) 
\mapsto  \langle {\bf u}\rangle$, of a fibre bundle 
with fibre $W_{n,k}$. This is just the projection of the $W_{n,k}$-bundle associated to the universal 
$n$-plane bundle $\gamma_{\infty, n}$.   Clearly $\pi\circ
 \wt{f}=f\circ q$ and $f^*(\gamma_{\infty, n})\cong n\gamma$.   Thus the following diagram commutes:
\[
\begin{array}{ccc}
W' & \stackrel{{\wt f}}{\lr} & BU(n-k)\\
q\downarrow &&\downarrow \pi\\
L&\stackrel{f}{\lr}& BU(n).\\
\end{array}
\]
The $W_{n,k}$-bundle $W'\lr L$ is the pull-back 
of the bundle $BU(n-k)\lr BU(n)$.  In particular the 
former bundle is $\bz$-orientable.
We consider the Serre spectral sequence of the $W_{n,k}$-bundle $q:W'\lr L$ which 
converges to  $H^*(W';\bz)\cong H^*(W_{n,k;m};\bz)$.    
We have $E_2^{p,q}=H^p(L;H^q(W_{n,k};\bz))
=H^p(L^\infty(m);\bz)\otimes H^q(W_{n,k};\bz)$ since 
$H^*(W_{n,k};\bz)=\Lambda_\bz(y_{2n-2k+1},\ldots, y_{2n-1})$ is free abelian.
It is well-known that $H^*(L^\infty(m);\bz)
\cong \bz[y_2]/\langle my_2\rangle$.  By comparing the 
Serre spectral sequence of $\pi:BU(n-k)\lr BU(n)$, we 
see that the cohomology classes $y_{2n-2k+2j-1}, 1\leq j\leq k,$ are transgressive. Indeed $\tau(y_{2n-2k+2j-1})
=c_{2n-2k+2j}(n\gamma)\in H^*(L;\bz), 1\leq j\leq k$.
That is $\tau(y_{2n-2k+2j-1})={n\choose k-j}y_2^{n-k+j}$.  It follows that ${n\choose j}y_2^r=0$ in $H^{2r}(W';\bz)$ 
for $n-k<j\leq r$ and so, since $my_2=0$,
 we see that the order of $y_2^{r}\in H^{2n-2k+2j}(W';\bz)$ equals $m_r$. 
In particular, the height $h$ of $y_2$ is as stated in the theorem.  \end{proof}

A complete description of the ring structure of $H^*(W_{n,k;m};\bz)$ appears to be more intricate.  
However, it is clear from the above 
proof that the torsion subgroup in $H^*(W_{n,k;m};\bz)$ 
is generated by the $y_2^j, 1\leq j< h$.  Also, it can be seen readily that there exist classes 
$v_{2n-2k+2j-1}\in H^*(W_{n,k;m};\bz)$, $ 1\leq j\leq k,$ which generate 
a free abelian group of rank $k$. Furthermore, their  
reduction mod any prime $p$ not dividing $m$ are 
the generators of $H^*(W_{n,k;m};\bz_p)$ given in 
Theorem \ref{cohomology}(i).  They arise from the generators of the 
kernel of the transgression in the spectral sequence 
in the above proof.  

As an application we have the following theorem.  We write $p(M)$ (resp. $w(M)$ for the 
total Pontrjagin class (resp. total Stiefel-Whitney class) of a differentiable manifold $M$.  
(See \cite{ms}.)

\begin{proposition}  \label{pontrjagin}
Let $2\leq k\leq n-2.$ and let $m\geq 2$.
One has  $p(W_{n,k;m})=(1+y_2^2)^{nk}$ for all $r\geq 1$.  The total Stiefel-Whitney class 
$w(W_{n,k;m})=(1+y_1^2)^{nk}$, where it is understood that $y_1=0$ when $m$ is odd.  \\
\end{proposition}
\begin{proof}
Consider the 
complexified tangent bundle $\tau_\bc:=\tau W_{n,k;m}\otimes_\br\varepsilon_\bc$.  From (4), $\tau_\bc$ 
is stably equivalent to the complex vector bundle 
$nk(\xi_{n,k;m}\oplus \xi_{n,k;m}^\vee)$. Therefore 
$c(\tau_\bc)=(1+y_2)^{nk}(1-y_2)^{nk}=(1-y_2^2)^{nk}$.   It follows that the $j$-th Pontrjagin class $p_j(W_{n,k;m}) 
={nk\choose j}y_2^{2j}$.

Using (4) we get 
 $w(W_{n,k;m})=w(\xi_{n,k;m})^{nk}=(1+y_1^2)^{nk}$.
 \end{proof}

Recall from Theorem \ref{span} that $W_{n,n-1;m}$ is parallelizable for all $m$.  The rest of the $W_{n,k;m}$ 
are \textit{not} stably parallelizable for most values of $m$.

\begin{theorem}\label{sparallelizability}
Let $1<k\leq n-2$ and $m\geq 2$.     
If there exists an $r\geq 1$ such that ${nk\choose r}$ is not divisible by $m_{2r}$, then 
$W_{n,k;m}$ is not stably parallelizable.   In particular, if $W_{n,k;m}$ is stably parallelizable,
then $m$ divides $nk$. 
\end{theorem}

\begin{proof}
 If ${nk\choose r}$ is not divisible by $m_{2r}$, then $m_{2r}>1$ and so $h>2r$. Therefore 
 $p_r(W_{n,k;m})={nk\choose r} y^{2r} \neq 0$ by Proposition \ref{pontrjagin}.   It follows that $W_{n,k;m}$ 
is not stably parallelizable (cf. \cite[Lemma 15.2]{ms}).    
 
 As for the second assertion, since $h>n-k\geq 2,$ one has $y_2^2\neq 0$. If $W_{n,k;m}$ is stably parallelizable, 
then $p_1(W_{n,k;m})=nky_2^2=0$ and hence $m_2=m$ divides $nk$.   
\end{proof}

\begin{remark}  {\em 
 The above theorem does not settle completely the question of stable parallelizability of 
$W_{n,k;m}$.  Suppose that  $n, k$ are powers of a prime $p$ and $m=p$.  Then $h=n$ and $m_r=p, \forall  r<n$.  
In this case, $p_j(W_{n,k;m})=0, w_j(W_{n,k;m})=0$ for all $j>0$.   
We remark that in the case of lens spaces $L^n(p)$ where $p$ is an odd prime,  Kambe \cite{kambe} has 
obtained non-immersion results using K-theory calculations. 
 Combined with the work of Adams \cite{adams} on the order of  the Hopf bundle 
$\xi_{n,1;2}$, one obtains that for a fixed $n$,  for all but finitely many $m>1$, 
the lens spaces are not stably parallelizable.  
}
\end{remark}

\noindent
{\bf Acknowledgment:} We thank J\'ulius Korba\v{s} for a careful reading of an earlier 
version of this paper and for his valuable comments. Also 
we  thank the referees for their valuable comments and suggestions.  We owe Remark 2.6 and an improved 
Theorem 2.4(i) to one of them.  The first named author thanks the Institute of 
Mathematical Sciences, Chennai, for its hospitality 
where this work was carried out.



\begin{thebibliography}{99}
\bibitem{adams} ADAMS, J. F.: \textit{Vector fields on spheres,} Ann.  Math.  \textbf{75}, (1962), 603--632.

\bibitem{agmp} ASTEY, L.--- GITLER, S.--- MICHA, E. ---PASTOR, G:
\textit{Cohomology of complex projective Stiefel manifolds}, 
Canad. J. Math. \textbf{51}, (1999), 897--914.

\bibitem{bhlsz} BARUFATTI, N.--- HACON, D--- LAM, K.-Y.---
SANKARAN, P.---ZVENGROWSKI, P.:  \textit{The order of real line bundles}, Bol. Soc. Mat. Mexicana  \textbf{10} (2004), 149--158. 
 
\bibitem{borel} BOREL, A: \textit{Sur la cohomologie des espaces fibr\'es principaux et des espaces homog\`enes de groupes de Lie compacts},  
Ann.  Math. \textbf{57}, (1953), 115--207. 

\bibitem{gh} GITLER, S.---HANDEL, D.: \textit{The projective Stiefel manifolds-I}, Topology \textbf{7}, (1968), 39--46.
\bibitem{hatcher} HATCHER, A.: \textit{Algebraic topology}, Cambridge University Press, 2002.
\bibitem{husemoller} HUSEMOLLER, D. \textit{Fibre bundles} 
Third Edition, Grad. Texts in Math. \textbf{20}, Springer-Verlag, N.Y. 1994.
\bibitem{iwata} IWATA, K.: \textit{Span of lens spaces}, Proc. Amer. Math. Soc. \textbf{26}, (1970), 687--688.

\bibitem{kambe} KAMBE, T.: \textit{The structure of $K_\Lambda$-rings of the lens space and their applications} J. Math. Soc. Japan \textbf{18}, (1966), 135--146.
\bibitem{kz94} KORBA\v{S}, J.---ZVENGROWSKI, P.: \textit{The vector field problem: a survey with emphasis on specific manifolds}, Exposition. Math. \textbf{12}, (1994),  3--20.
\bibitem{kz96} KORBA\v{S}, J.--- ZVENGROWSKI, P.:  \textit{On sectioning tangent bundles and other vector bundles},  Rend. Circ. Mat. Palermo Suppl. \textbf{96}, (1996), 85--104. 
\bibitem{koschorke}  KOSCHORKE, U.:\textit{Vector fields and other vector bundle morphisms--a singularity approach}, 
Lecture Notes in Mathematics, \textbf{847}, Springer, Berlin, 1981.
\bibitem{lam1} LAM, K.-Y.: \textit{A formula for the tangent bundle of flag manifolds and related manifolds}, Trans. Amer. Math. Soc.  \textbf{213}, (1975), 305--314.
\bibitem{lam2}  LAM, K.-Y.: \textit{Sectioning vector bundles over real projective spaces,} Quart. J. Math. \textbf{23}, (1972), 97--106.

\bibitem{dl} LAM, K.-Y.---RANDALL, D.: 
\textit{Periodicity of the 
geometric dimension for real projective spaces}, In: Algebraic Topology: New Trends in Localization and 
Periodicity, Prog. Math. \textbf{136}, (1996), 223--235. Brikh\"auser, Basel.

\bibitem{ms} MILNOR, J. W.; STASHEFF, J. D.:
\textit{Characteristic classes},  
Annals of Mathematics Studies, {\bf 76}, Princeton University Press, Princeton, N. J. 1974.
\bibitem{wang} WANG, H.-C.: \textit{Closed manifolds with homogeneous complex structure}, Amer. J. Math. \textbf{76}, (1954), 1--32.
\bibitem{yoshida} YOSHIDA, T.: \textit{A remark on vector fields on lens spaces},
J. Sci. Hiroshima Univ. Ser. A-I Math. \textbf{31}, (1967), 13--15.
\end{thebibliography}
\end{document}